\documentclass{amsart}
\usepackage[normalem]{ulem}
\usepackage{stmaryrd}
\usepackage{mathtools}
\usepackage{amssymb,epsfig,xspace}
\usepackage{ifpdf}
\usepackage{color}
\usepackage{bm}
\usepackage{booktabs}

\usepackage{tgpagella}

\usepackage{bbold}
\usepackage{mathrsfs}
\usepackage{esint}

\usepackage{enumerate}
\usepackage{soul}
\usepackage{commath}

\newtheorem{theorem}{Theorem}[section]

\newtheorem{proposition}[theorem]{Proposition}

\theoremstyle{definition}

\theoremstyle{remark}
\newtheorem{remark}[theorem]{Remark}

\numberwithin{equation}{section}

\newcommand{\1}{\mathbb 1}

\newcommand{\eps}{\varepsilon}

\newcommand{\R}{\mathbb{R}}

\newcommand{\be}{\begin{equation}}
\newcommand{\ee}{\end{equation}}

\newcommand{\new}[1]{{#1}}

\DeclareFontFamily{U}{mathx}{\hyphenchar\font45}
\DeclareFontShape{U}{mathx}{m}{n}{
      <5> <6> <7> <8> <9> <10>
      <10.95> <12> <14.4> <17.28> <20.74> <24.88>
      mathx10
      }{}
\DeclareSymbolFont{mathx}{U}{mathx}{m}{n}
\DeclareFontSubstitution{U}{mathx}{m}{n}
\DeclareMathAccent{\widecheck}{0}{mathx}{"71}
\DeclareMathAccent{\wideparen}{0}{mathx}{"75}

\newcommand{\uumlaut}{{\"u}}

\newcommand{\tria}{{\mathcal T}}
\newcommand{\bbT}{\mathbb{T}}

\newcommand{\cP}{\mathcal{P}}

\newcommand{\identity}{\mathrm{Id}}

\newcommand{\HH}{\mathscr{H}}
\newcommand{\V}{\mathscr{V}}

\newcommand{\B}{\mathscr{B}}
\newcommand{\Ss}{\mathscr{S}}
\newcommand{\W}{\mathscr{W}}
\newcommand{\Y}{\mathscr{Y}}
\newcommand{\ZZ}{\mathscr{Z}}

\DeclareMathOperator{\ran}{ran}

\DeclareMathOperator{\diag}{diag}
\DeclareMathOperator{\gen}{gen}

\newcommand{\cL}{\mathcal L}
\newcommand{\Lis}{\cL\mathrm{is}}
\newcommand{\cLis}{\cL\mathrm{is}_c}

\title{Uniform preconditioners of linear complexity for problems of negative order}

\date{\today}

\author{Rob Stevenson, Raymond van Veneti\"{e}}
\address{
Korteweg-de Vries Institute for Mathematics,
University of Amsterdam,
P.O. Box 94248,
1090 GE Amsterdam, The Netherlands}
\email{r.p.stevenson@uva.nl, r.vanvenetie@uva.nl}

\thanks{The second author has been supported by the Netherlands Organization for Scientific Research
(NWO) under contract. no. 613.001.652}



\begin{document}
\begin{abstract}
    We propose a multi-level type operator that can be used in the framework of operator (or Cald\'{e}ron) preconditioning to construct uniform preconditioners for negative order operators discretized by piecewise polynomials on a family of possibly locally refined partitions.
    The cost of applying this multi-level operator scales linearly in the number of mesh cells. Therefore, it provides a uniform preconditioner that can be applied
in linear complexity when used within the preconditioning framework from our earlier work [\emph{Math.~of Comp.}, 322(89) (2020), pp. 645--674].
\end{abstract}

\maketitle

\section{Introduction}
In this work we construct a multi-level type preconditioner for operators of
negative orders $-2s \in [-2,0]$  that can be applied in linear time and yields uniformly bounded condition numbers.
The preconditioner will be constructed using the framework of `operator preconditioning'
studied earlier in e.g.~\cite{249.15,38.76,138.26,249.97}.
The role of the `opposite order operator' will be fulfilled by a multi-level type operator, based on the work of Wu and Zheng in~\cite{316.55}.

For some $d$-dimensional domain (or manifold) $\Omega$, a measurable, closed, possibly empty $\gamma \subset \partial\Omega$, and an $s \in [0,1]$,
we consider the Sobolev spaces
\[
\W:=[L_2(\Omega),H^1_{0,\gamma}(\Omega)]_{s,2},\quad \V:=\W'.
\]
with $H^1_{0,\gamma}(\Omega)$ being the closure in $H^1(\Omega)$ of the smooth functions on $\Omega$ that vanish at $\gamma$.
Let $(\V_\tria)_{\tria \in \bbT} \subset \V$ be a family of \emph{piecewise} or \emph{continuous piecewise} polynomials of some fixed degree w.r.t.~uniformly shape regular, possibly locally refined partitions.
With, for $\tria \in \bbT$,
$A_\tria \colon \V_\tria \to \V_\tria'$ being some boundedly invertible linear operator,
we are interested in constructing a \emph{preconditioner} $G_\tria \colon \V_\tria' \to \V_\tria$ such that the preconditioned operator $G_\tria A_\tria\colon \V_\tria \to \V_\tria$ is uniformly boundedly invertible, and an application of $G_\tria$ can be evaluated in $\mathcal O(\dim \V_\tria)$ arithmetic operations.

In order to create such a preconditioner, we will use the framework described in our earlier work~\cite{249.97}.
Given $\V_\tria$, we constructed an auxiliary space $\W_\tria \subset \W$ with $\dim \W_\tria = \dim \V_\tria$, such that for $D_\tria$ defined by $(D_\tria \new{v})(w) := \langle v, w \rangle_{L_2(\Omega)} (v \in \V_\tria, w \in \W_\tria)$ and some suitable `opposite order' operator $B^\W_\tria \colon \W_\tria \to \W_\tria'$, a preconditioner $G_\tria$ of the form $G_\tria := D_\tria^{-1} B^\W_\tria (D_\tria')^{-1}$ is found.
The space $\W_\tria$ is equipped with a basis that, \new{modulo a scaling,} is biorthogonal to the canonical basis for $\V_\tria$, \new{so that the representation of $D_\tria$ is an invertible diagonal matrix.}

With $\Ss_{\tria,0}^{0,1} \subset \W$ being the space
of continuous piecewise linears w.r.t.~$\tria$, zero on $\gamma$, the
above preconditioning approach hinges on the availability of a uniformly boundedly invertible operator $B_\tria^{\Ss}\colon \Ss_{\tria,0}^{0,1} \rightarrow (\Ss_{\tria,0}^{0,1})'$, which is generally the most demanding requirement.
For example, if $s =\frac{1}{2}$ and $\gamma = \emptyset$, a viable option is to take $B_\tria^{\Ss}$ as the discretized hypersingular operator.
While this induces a uniform preconditioner, the application of $B_\tria^\Ss$ cannot be evaluated in linear complexity.

In this work we construct a suitable multi-level type operator $B_\tria^{\Ss}$ that \emph{can} be applied in linear complexity.
For this construction we require $\bbT$ to be a family of conforming partitions created by Newest Vertex Bisection ({\cite{203.1,286}).
In the aforementioned setting of having an arbitrary $s \in [0,1]$, this multi-level operator $B_\tria^{\Ss}$ induces a uniform preconditioner $G_\tria$, i.e., $G_\tria A_\tria$ is uniformly well-conditioned, where the cost of applying $G_\tria$ scales linearly in $\dim \V_\tria$. We also show that the preconditioner extends to the more general \emph{manifold} case, where $\Omega$ is a $d$-dimensional (piecewise) smooth Lipschitz manifold, and the trial space $\V_\tria$ is
the parametric lift of \new{a} space of piecewise or continuous piecewise polynomials.

\new{Finally, we remark} that common multi-level preconditioners based on overlapping subspace decompositions are known not to work well for operators of negative order.
A solution is provided by resorting to direct sum multi-level subspace decompositions. Examples are given by wavelet preconditioners,
or closely related, the preconditioners from~\cite{34.8}, \new{for the latter} assuming quasi-uniform partitions.

For $\new{-}s = \new{-}\frac12$, an optimal multi-level preconditioner based on a \new{\emph{non}}-overlapping subspace decomposition
for operators  defined on the boundary of a $2$- or $3$-dimen\-sional Lipschitz polyhedron was recently introduced in~\cite{75.069}.
\subsection{Outline}
In Sect.~\ref{sec:preconditioning} we summarize the (operator) preconditioning framework
from~\cite{249.97}. In Sect.~\ref{sec:multilevel} we provide the multi-level type operator
that can be used as the `opposite order' operator inside the preconditioner framework. In Sect.~\ref{sec:manifold} we comment on how to generalize the results to the case of piecewise smooth manifolds. In Sect.~\ref{sec:numerics} we conclude with numerical
results.

\subsection{Notation}
In this work, by $\lambda \lesssim \mu$ we mean that $\lambda$ can be bounded by a multiple of $\mu$, independently of parameters which $\lambda$ and $\mu$ may depend on, with the sole exception of the space dimension $d$, or in the manifold case, on the parametrization of the manifold that is used to define the finite element spaces on it. Obviously, $\lambda \gtrsim \mu$ is defined as $\mu \lesssim \lambda$, and $\lambda\eqsim \mu$ as $\lambda\lesssim \mu$ and $\lambda \gtrsim \mu$.

For normed linear spaces $\Y$ and $\ZZ$, in this paper for convenience over $\R$, $\cL(\Y,\ZZ)$ will denote the space of bounded linear mappings $\Y \rightarrow \ZZ$ endowed with the operator norm $\|\cdot\|_{\cL(\Y,\ZZ)}$. The subset of invertible operators in $\cL(\Y,\ZZ)$  with inverses in $\cL(\ZZ,\Y)$
will be denoted as $\Lis(\Y,\ZZ)$.
The \emph{condition number} of a $C \in \Lis(\Y,\ZZ)$ is defined as $\kappa_{\Y,\ZZ}(C):=\|C\|_{\cL(\Y,\ZZ)}\|C^{-1}\|_{\cL(\ZZ,\Y)}$.

For $\Y$ a reflexive Banach space and $C \in \cL(\Y,\Y')$ being \emph{coercive}, i.e.,
\[
\inf_{0 \neq y \in \Y} \frac{(Cy)(y)}{\|y\|^2_\Y} >0,
\]
both $C$ and $\Re(C)\!:= \!\frac{1}{2}(C+C')$ are in $\Lis(\Y,\Y')$ with
\begin{align*}
\|\Re(C)\|_{\cL(\Y,\Y')} &\leq \|C\|_{\cL(\Y,\Y')},\\
\|C^{-1}\|_{\cL(\Y',\Y)} & \leq \|\Re(C)^{-1}\|_{\cL(\Y',\Y)}=\Big(\inf_{0 \neq y \in \Y} \frac{(Cy)(y)}{\|y\|^2_\Y}\Big)^{-1}.
\end{align*}
The set of coercive $C \in \Lis(\Y,\Y')$ is denoted as $\cLis(\Y,\Y')$.
If $C   \in \cLis(\Y,\Y')$, then $C^{-1} \in \cLis(\Y',\Y)$ and $\|\Re(C^{-1})^{-1}\|_{\cL(\Y,\Y')} \leq \|C\|_{\cL(\Y,\Y')}^2 \|\Re(C)^{-1}\|_{\cL(\Y',\Y)}$.

Given a family of operators $C_i \in \Lis(\Y_i,\ZZ_i)$ ($\cLis(\Y_i,\ZZ_i)$), we will write $C_i \in \Lis(\Y_i,\ZZ_i)$ ($\cLis(\Y_i,\ZZ_i)$) uniformly in $i$, or simply `uniform', when
    \[
    \sup_{i} \max(\|C_i\|_{\cL(\Y_i,\ZZ_i)},\|C_i^{-1}\|_{\cL(\ZZ_i,\Y_i)})<\infty,
     \]
or
    \[
    \sup_{i} \max(\|C_i\|_{\cL(\Y_i,\ZZ_i)},\|\Re(C_i)^{-1}\|_{\cL(\ZZ_i,\Y_i)})<\infty.
     \]
\section{Preconditioning}\label{sec:preconditioning}
Let $(\tria)_{\tria \in \bbT}$ be a family of \emph{conforming} partitions of a domain $\Omega \subset \R^d$ into (open) \emph{uniformly shape regular} $d$-simplices, where we assume that $\gamma$ is the (possibly empty) union of $(d-1)$-faces of $T \in \tria$.
For $d \geq 2$, such partitions automatically satisfy a uniform $K$-mesh property, and for $d=1$ we impose this as an additional condition.
The discussion of the manifold case is postponed to Sect.~\ref{sec:manifold}.

Recalling that $\V_\tria \subset \V$ is a family of piecewise or continuous piecewise polynomials of some fixed degree w.r.t.~$\tria$,
let $A_\tria \in \Lis(\V_{\tria},\V_{\tria}')$ \new{uniformly} in $\tria \in \bbT$.
A common setting is that $(A_\tria v)(\tilde v):=(A v)(\tilde v)$ ($v, \tilde v \in \V_\tria$) for some $A \in \cLis(\V, \V')$.
We are interested in finding optimal \emph{preconditioners} $G_\tria$ for $A_\tria$, i.e., $G_\tria \in \Lis(\V_{\tria}',\V_{\tria})$ uniformly in $\tria \in \bbT$,
whose application moreover requires ${\mathcal O}(\dim \V_\tria)$ arithmetic operations.

Recall the space
 \[
     \Ss_{\tria,0}^{0,1}:=\{u \in H^1_{0,\gamma}(\Omega)\colon u|_T \in \cP_1 \,(T \in \tria)\} \subset \W
\]
(thus equipped with $\|\cdot\|_\W$). In~\cite{249.97}, using operator preconditioning, we reduced the issue of constructing such preconditioners $G_\tria$ to the issue of constructing
$B_\tria^{\Ss} \in \cLis(\Ss_{\tria,0}^{0,1},(\Ss_{\tria,0}^{0,1})')$ uniformly. In the next section we summarize this reduction.

\subsection{Construction of optimal preconditioners}\label{sec:theory}
For the moment, consider the lowest order case of $\V_\tria$ being either the space of piecewise constants or continuous piecewise linears. In~\cite{249.97} a space $\W_\tria \subset \W$ was constructed with $\dim \W_\tria=\dim \V_\tria$ and
\begin{equation} \label{infinfsup}
\inf_{\tria \in \bbT} \inf_{0 \neq v \in \V_\tria} \sup_{0 \neq w \in \W_\tria} \frac{\langle v,w\rangle_{L_2(\Omega)}}{\|v\|_{\V} \|w\|_{\W}}>0.
\end{equation}
Moreover, $\W_\tria \subset \W$ was equipped with a locally supported basis $\Psi_\tria$ that, \new{modulo a scaling}, is $L_2(\Omega)$-\emph{biorthogonal} to the canonical basis $\Xi_\tria$ of $\V_\tria$.

As a consequence of~\eqref{infinfsup}, $D_\tria$ defined by $(D_\tria v)(w):=\langle v,w\rangle_{L_2(\Omega)}$ ($v \in \V_\tria,\, w \in \W_\tria$)
is in $\Lis(\V_{\tria},\W_{\tria}')$ uniformly.
We infer that once we have constructed $B^\W_\tria \in \Lis(\W_{\tria},\W_{\tria}')$ uniformly, \new{then} by taking
\begin{equation} \label{precond}
G_\tria:=D_\tria^{-1} B^\W_\tria (D'_\tria)^{-1},
\end{equation}
we have $G_\tria \in \Lis(\V_\tria', \V_\tria)$ uniformly.
Biorthogonality, modulo a scaling, of the bases $\Psi_\tria$ and $\Xi_\tria$
implies that the matrix representation of $D_\tria$ is diagonal, so that $D^{-1}_\tria$ and its adjoint can be applied
 in linear complexity.

The aforementioned space $\W_\tria$ is a subspace of $\Ss_{\tria,0}^{0,1} \oplus \B_\tria \subset \W$,
where $\B_\tria$  is a `bubble space'  with $\dim \B_\tria={\mathcal O}(\#\tria)$, \new{such that}
the projector $I_\tria$ on $\Ss_{\tria,0}^{0,1} \oplus \B_\tria$, defined by $\ran I_\tria=\Ss_{\tria,0}^{0,1}$ and $\ran(\identity-I_\tria)=\B_\tria$,
is `local' and uniformly bounded, and the canonical basis $\Theta_\tria$ of `bubbles' for $\B_\tria$ is, \new{when normalized in $\|\cdot\|_{\W}$}, a uniformly Riesz basis for $\B_\tria$.
\new{Because of the latter, $B^\B_\tria$ defined by
$$
(B^\B_\tria {\bf c}^\top \Theta_\tria)({\bf d}^\top \Theta_\tria):=\beta (\bm{\Delta}_\tria {\bf c})^\top {\bf d}
$$ for some diagonal $\bm{\Delta}_\tria  \eqsim \diag(\langle \Theta_\tria,\Theta_\tria\rangle_\W)$ and constant $\beta>0$ is in $\cLis(\B_{\tria},\B_{\tria}')$ uniformly.}

Given some `opposite order' operator $B_\tria^{\Ss} \in \cLis(\Ss_{\tria,0}^{0,1},(\Ss_{\tria,0}^{0,1})')$, by taking
\begin{equation}\label{eq:bubble_splitting}
B^\W_\tria:=I_\tria' B_\tria^{\Ss} I_\tria +(\identity-I_\tria)' B^\B_\tria (\identity-I_\tria),
\end{equation}
\new{it holds that} $B^\W_\tria \in \cLis(\W_{\tria},\W_{\tria}')$ uniformly (\cite[Prop.~5.1]{249.975}), which makes $G_\tria$ a uniform preconditioner.

\subsection{Implementation of $G_\tria$}\label{sec:implementation}
Recalling the aforementioned bases $\Xi_\tria$, $\Psi_\tria$, and $\Theta_\tria$ for $\V_\tria$, $\W_\tria$ and $\B_\tria$, respectively, equipping $\Ss_{\tria,0}^{0,1}$ with the nodal basis $\Phi_\tria$, and equipping
$\V_\tria'$, $\W_\tria'$, $\B_\tria'$, and $(\Ss_{\tria,0}^{0,1})'$ with the dual bases $\Xi_\tria'$, $\Psi_\tria'$, $\Theta_\tria'$, and $\Phi_\tria'$, respectively,
the representation of $A_\tria \in \cL(\V_\tria, \V_\tria')$ is the stiffness matrix $\bm{A}_\tria:=(A_\tria \Xi_\tria)(\Xi_\tria):=[(A_\tria \eta)(\xi)]_{(\xi,\eta) \in \Xi_\tria}$, and
the representation of $G_\tria \in \cL(\V_\tria', \V_\tria)$ is the matrix $\bm{G}_\tria:=(G \Xi'_\tria)(\Xi'_\tria)$.
It is given by
\begin{equation}\label{eq:repr_precon}
\bm{G}_\tria=\bm{D}_\tria^{-1} \big(\bm{p}_\tria^\top \bm{B}_\tria^{\Ss} \bm{p}_\tria+\bm{q}_\tria^\top \bm{B}^{\B}_\tria \bm{q}_\tria\big) \bm{D}_\tria^{-\top},
\end{equation}
where both
\begin{alignat*}{2}
\bm{D}_\tria& :=(D_\tria \Xi_\tria)(\Psi_\tria),& \quad \bm{B}^{\B}_\tria&:=(B^{\B}_\tria \Theta_\tria)(\Theta_\tria)
\intertext{are \emph{diagonal}, both}
\bm{p}_\tria &:= (I_\tria \Psi_\tria)(\Phi_\tria'),& \bm{q}_\tria&:= ((\identity-I_\tria) \Psi_\tria)(\Theta_\tria')
\end{alignat*}
are \emph{uniformly sparse}, and
\begin{equation}\label{eq:repr_opposite}
\bm{B}_\tria^{\Ss}:=(B_\tria^{\Ss}\Phi_\tria)(\Phi_\tria).
\end{equation}
Note that the cost of the application of $\bm{G}_\tria$ scales \emph{linearly} in $\#\tria$ as soon as this holds true for the application of $\bm{B}_\tria^{\Ss}$.

The above preconditioning approach is summarized in the following theorem.
\begin{theorem}[{\cite[Sect.~3]{249.97}}]\label{assumption}
Given a family $B_\tria^{\Ss} \in \cLis(\Ss_{\tria,0}^{0,1},(\Ss_{\tria,0}^{0,1})')$ uniformly in $\tria \in \bbT$.
Then for $B^\W_\tria$ as described in~\eqref{eq:bubble_splitting},
the operator $G_\tria$ from~\eqref{precond} is a uniform preconditioner.
Furthermore, if the matrix representation $\bm{B}_\tria^{\Ss}$, cf.~\eqref{eq:repr_opposite}, can be applied in $\mathcal{O}(\#\tria)$ operations, then
the matrix representation of the preconditioner $\bm{G}_\tria$, cf.~\eqref{eq:repr_precon}, can be applied in $\mathcal{O}(\#\tria)$ operations.
\end{theorem}

\new{Because $B_\tria^\W$ in \eqref{eq:bubble_splitting} is given as the sum of two operators that `act' on different subspaces of $\W_\tria$, the condition number of the preconditioned system depends on the relative scaling of both these operators which can be steered by selecting the parameter $\beta$.
A suitable $\beta$ will be selected experimentally.}

\new{Alternatively, \cite[Prop.~5.1]{249.975} shows that a value of $\beta$ is reasonable if it is chosen such that the interval bounded by the coercivity and boundedness constants of $B_\tria^{\Ss} $ is included in that interval corresponding to $B_\tria^{\B} $ or vice versa. Also these coercivity and boundedness constants can be approximated experimentally or by making some theoretical estimates.}

\new{Constructions of $\Psi_\tria$, $\Theta_\tria$, and $\bm{\Delta}_\tria$, and} resulting explicit formulas for matrices $\bm{D}_\tria$, $\bm{B}^{\B}_\tria$, $\bm{p}_\tria$, $\bm{q}_\tria$ are derived in~\cite{249.97}. For ease of reading we \new{recall} these
formulas below for the case \new{that} $\V_\tria$ is the space of \emph{piecewise constants}. \new{F}or the continuous piecewise linear case we refer to~\cite[\new{Sect.~4.2}]{249.97}.

\subsubsection{Piecewise constant trial space $\V_\tria$} \label{sec:piecewise_contant}
 For $\tria \in \bbT$,  we define $N_\tria$ as the set of vertices of $\tria$, and $N_\tria^{0}$ as the set of vertices of $\tria$ that are not on $\gamma$. For $\nu \in N_\tria$ we set its \emph{valence}
 \[
d_{\tria,\nu}:=\#\{T \in \tria\colon \nu \in \overline{T}\}.
\]
For $T \in \tria$, and with $N_T$ denoting the set of its vertices, we set $N^0_{\tria,T}:=N^0_\tria \cap N_T$.

If one considers $\V_\tria$ as the space of discontinuous piecewise constants, i.e.
\[
    \V_\tria = \Ss_\tria^{-1,0} := \{ u \in L_2(\Omega) \colon u|_T \in P_0 (T \in \tria)\} \subset \V,
\]
equipped with the canonical basis $\Xi_\tria := \{ \1_T \colon T \in \tria\}$, then we find, \new{for arbitrary constant $\beta > 0$},
 \begin{align*}
&\bm{D}_\tria=\diag\{|T|\colon T \in \tria\}, \quad &&
(\bm{p}_\tria)_{\nu T}=\left\{\begin{array}{cl} d_{\tria,\nu}^{-1} & \text{if } \nu \in N^0_{\tria,T},\\ 0 & \text{if }\nu \not\in N^0_{\tria,T},\end{array}\right.\\
&\bm{B}^{\B}_\tria=  \beta \bm{D}_\tria^{1-\frac{2s}{d}},\quad
&&\displaystyle (\bm{q}_\tria)_{T' T}=\delta_{T' T}- {\textstyle \frac{1}{d+1}}\hspace*{-1em}\sum_{\nu \in N^0_{\tria,T} \cap N^0_{\tria,T'}} d_{\tria,\nu}^{-1}.
\end{align*}
%
\subsection{Higher order case}
For higher order discontinuous or continuous finite element spaces $\V_\tria$,  suitable preconditioners $G_\tria$ can be built either from the current preconditioner $G_\tria$ for the lowest order case by application of a subspace correction method (most conveniently in the discontinuous case where on each element the space of polynomials of some fixed degree is split into the space of constants and its orthogonal complement), or by expanding $\W_\tria$ by enlarging the bubble space $\B_\tria$.
\new{While referring to~\cite{249.97} for details, we recall that with either option the construction of an optimal preconditioner $G_\tria$ that can be applied in linear complexity hinges on the availability of
operator $B_\tria^{\Ss} \in \cLis(\Ss^{0,1}_{\tria,0},(\Ss^{0,1}_{\tria,0})')$ uniform, that can be applied in linear complexity.}

\section{An operator $B_\tria^{\Ss}$ of multi-level type}\label{sec:multilevel}
In this section we will introduce an  operator $B_\tria^{\Ss} \in \cLis(\Ss^{0,1}_{\tria,0},(\Ss^{0,1}_{\tria,0})')$ of multi-level type.
The operator $B_\tria^{\Ss}$ is based on a stable multi-level decomposition of $\Ss_{\tria,0}^{0,1}$ given by Wu and Zheng~\cite{316.55}.
Usually such a stable multi-level decomposition is used as a theoretical tool for proving optimality
of an additive (or multiplicative) Schwarz type preconditioner
for an operator in $\cLis(\Ss_{\tria,0}^{0,1},(\Ss_{\tria,0}^{0,1})')$.
In this work, however, we are going to use their results
for the construction of
 the operator $B_\tria^{\Ss} \in \cLis(\Ss^{0,1}_{\tria,0},(\Ss^{0,1}_{\tria,0})')$ for which it is crucial that its application  can be implemented in
 linear complexity.

\subsection{Definition and analysis of $B_\tria^{\Ss}$}
For $d \geq 2$, let $\bbT$ be the family of all conforming partitions of $\Omega$ into $d$-simplices that can be created by Newest Vertex Bisection
starting from some given conforming initial partition $\tria_\bot$ that satisfies a \emph{matching condition} (\cite{249.87}).

With $\mathfrak{T}:=\cup_{\tria \in \bbT} \{T\colon T \in \tria\}$ and
$\mathfrak{N}:=\cup_{\tria \in \bbT} N_\tria$, for $T \in \mathfrak{T}$ let $\gen(T)$ be the number of bisections needed to create $T$ from its ancestor $T' \in \tria_\bot$, and for $\nu \in \mathfrak{N}$ let $\gen(\nu):=\min\{\gen(T)\colon T \in \mathfrak{T},\,\nu \in \new{N_T}\}$.
\new{Notice that $|T| \eqsim 2^{-\gen(T)}$.}
For $T \in \mathfrak{T}$, let $Q_T$ denote the $L_2(T)$-orthogonal projector onto $\cP_1(T)$.

The case $d=1$ can be included by letting $\bbT$ be the family of a partitions of $\Omega$ that can be constructed by bisections from $\tria_\bot=\{\Omega\}$ such that the generations of any two neighbouring subintervals in any $\tria \in \bbT$ differ by not more than one.

For $\tria \in \bbT$, set $L=L(\tria):=\max_{T \in \tria} \gen(T)$ and define
\[
\tria_\bot = \tria_0 \prec\tria_1 \prec \cdots \prec \tria_L=\tria \subset \bbT
\]
by constructing $\tria_{j-1}$ from $\tria_j$ by removing all $\nu \in \new{N_j:=}N_{\tria_j}$ from the latter for which $\gen(\nu)=j$.
For $\nu \in \new{N_j^0:=}N_{\tria_j}^0$, we define $\new{\omega_j}(\nu)=\cup\{T \in \tria_j\colon \nu \in \new{N_T}\}$.

With this hierarchy of partitions, we define an averaging quasi-interpolator $\Pi_j \in \cL(\Ss_{\tria,0}^{0,1},\Ss_{\tria_j,0}^{0,1})$ by
\begin{equation}\label{eq:averaging}
    (\Pi_j u)(\nu):=\frac{\sum_{\{T \in \tria_j\colon \nu \in \new{N_T}\}} |T| (Q_T u)(\nu)}{\sum_{\{T \in \tria_j\colon \nu \in \new{N_T}\}} |T|} \quad (u \in \Ss_{\tria,0}^{0,1},\,\nu \in \new{N^0_j}).
\end{equation}
 Since $\Ss_{\tria_j,0}^{0,1}$ is a space of continuous piecewise linears, it indeed suffices to define $\Pi_j u$ at the vertices $\new{N_j^0}$.
Recall that $\Ss_{\tria, 0}^{0,1} \subset \W:=[L_2(\Omega),H^1_{0,\gamma}(\Omega)]_{s,2}$ \new{for some $s \in [0,1]$}. The next theorem shows that $\Pi_j$ induces a stable multi-level decomposition of~$\Ss_{\tria,0}^{0,1}$.
\begin{theorem}[{\cite[Lemma~3.7]{316.55}}]\label{WZ}
 For the averaging quasi-interpolator $\Pi_j$ from~\eqref{eq:averaging}, and $\Pi_{-1}:=0$,  it holds that
\[
\|u\|_{\W}^2 \eqsim \sum_{j=0}^L 4^{j s/d} \|(\Pi_j-\Pi_{j-1})u\|_{L_2(\Omega)}^2 \quad(u \in \Ss_{\tria,0}^{0,1}).
\]
\end{theorem}
\begin{proof}
 In \cite{316.55}, the inequality `$\gtrsim$' was proven for the case $s=1$, $d \in \{2,3\}$, and $\gamma=\partial\Omega$. The arguments, however, immediately extend to $s \in [0,1]$, $d \geq 1$, and $\gamma \subsetneq \partial\Omega$.

The proof of the other inequality `$\lesssim$' follows from well-known arguments:
For some $t \in (\new{1},\frac32)$, let $\HH^{r t}:=[L_2(\Omega),H^1_{0,\gamma}(\Omega) \cap H^t(\Omega)]_{r,2}$ for $r \in [0,1]$.
Then $\HH^{s} \simeq \W$ by the \emph{reiteration theorem}, and for $r \in [0,1]$,
$\|\cdot\|_{\HH^{r t}} \lesssim 2^{j r t/d} \|\cdot\|_{L_2(\Omega)}$ on $\Ss_{\tria_j,0}^{0,1}$.

Let $u \in \Ss_{\tria,0}^{0,1}$ be written as $\sum_{j=0}^L u_j$ with $u_j \in \Ss_{\tria_j,0}^{0,1}$.
Then for $\eps \in (0,s)$, $\eps\leq t-s$, we have
\begin{align*}
\|u\|^2_{\HH^{s}(\Omega)} &\lesssim \sum_{j=0}^L \sum_{i=j}^L \|u_j\|_{\HH^{s+\eps}(\Omega)}\|u_i\|_{\HH^{s-\eps}(\Omega)}\\
& \lesssim \sum_{j=0}^L \sum_{i=j}^L 2^{j (s+\eps)/d} 2^{i (s-\eps)/d} \|u_j\|_{L_2(\Omega)}\|u_i\|_{L_2(\Omega)} \lesssim \sum_{j=0}^L 4^{j s/d}\|u_j\|_{L_2(\Omega)}^2.\qedhere
\end{align*}
\end{proof}

The relevance of the multi-level decomposition from Theorem~\ref{WZ} by Wu and Zheng lies in the fact that
$(\Pi_j u)(\nu)$ can only differ from $(\Pi_{j-1}u)(\nu)$ in any $\nu \in \new{N^0_j}\setminus N^0_{\new{j-1}}$ as well as in \new{only \emph{two}\footnote{\new{As pointed out in \cite{316.55}, for $d \geq 3$ the number of such neighbours will be larger when employing the Scott-Zhang quasi-interpolator. Moreover, this interpolator is not suited for $s \leq \frac12$.}} of its neighbours in $N^0_{j-1}$ (the endpoints
of the edge on which $\nu$ was inserted):}

\begin{proposition}[{\cite[Lemma~3.1]{316.55}}]\label{propWZ}
With for $j \geq 1$,
\[
M^0_{\new{j}}:=\{\nu \in N^0_{\new{j-1}}\colon \new{\omega_j}(\nu)=\new{\omega_{j-1}}(\nu)\},
\]
 it holds that for $\nu \in M^0_{\new{j}}$, $((\Pi_j-\Pi_{j-1})u)(\nu)=0$, see Figure~\ref{fig2}.
\end{proposition}

\begin{figure}
\begin{center}
\includegraphics[width=2cm]{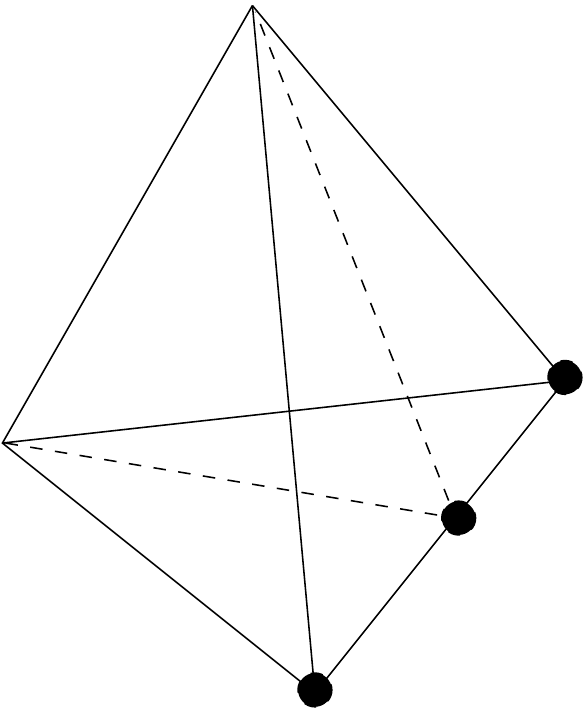}
\end{center}
\caption{For $d=3$, a tetrahedron $T \in \tria_{j-1}$ and its bisection. The dots indicate \new{all} vertices in $\new{N_j^0} \setminus M^0_{\new{j}}$.}
\label{fig2}
\end{figure}

\begin{remark} The proof from \cite{316.55} given for $d \in \{2,3\}$ generalizes to $d \geq 1$.
Indeed the arguments that are used are based on the fact
that the basis \new{for} $S_1(T)$ that is dual to the nodal basis takes equal values in all but one nodal point. This is a consequence of the fact that the mass matrix of the nodal basis \new{for} $S_1(T)$, and so its inverse, is invariant under permutations of the barycentric coordinates, which holds true in any dimension.
\end{remark}

As a consequence of Proposition~\ref{propWZ}, we have
\[
    \|(\Pi_j-\Pi_{j-1})u\|_{L_2(\Omega)}^2 \eqsim 2^{-j} \hspace{-12pt} \sum_{\nu \in \new{N_j^0} \setminus M^0_{\new{j}}} \hspace{-10pt} |((\Pi_j-\Pi_{j-1})u)(\nu)|^2.
\]
From Theorem~\ref{WZ}, we conclude that $B_{\tria}^{\Ss}\new{=(B_{\tria}^{\Ss})'}  \in \cLis(\Ss_{\tria,0}^{0,1},(\Ss_{\tria,0}^{0,1})')$ defined by
\begin{equation}\label{eq:multilevel_operator}
(B_\tria^{\Ss}u)(v):=\sum_{j=0}^L 2^{j (\frac{2s}{d}-1)}
\sum_{\nu \in \new{N_j^0} \setminus M^0_{\new{j}}} ((\Pi_j-\Pi_{j-1})u)(\nu)((\Pi_j-\Pi_{j-1})v)(\nu)
\end{equation}
is uniform, i.e.~
\begin{equation}\label{unif_bound}
\sup_{\tria \in \bbT} \max\Big(\|B_\tria^{\Ss}\|_{\cL(\Ss_{\tria,0}^{0,1},(\Ss_{\tria,0}^{0,1})')},\|\new{(B_\tria^{\Ss})^{-1}}\|_{\cL((\Ss_{\tria,0}^{0,1})',\Ss_{\tria,0}^{0,1})}\Big)<\infty.
\end{equation}

\subsection{Implementation of $B_\tria^{\Ss}$}\label{sec:impl}
Since the operator $\Pi_j$ is a weighted local $L_2(\Omega)$ projection, it allows for a natural implementation
by considering $\Ss_\tria^{-1,1}$, the space of discontinuous piecewise linears w.r.t.~$\tria$.
Recall the nodal basis $\Phi_\tria$ for $\Ss_{\tria,0}^{0,1}$,
 and equip $\Ss_{\tria}^{-1,1}$ with the element-wise nodal basis.

Denote $\bm{E}_\tria$ for the representation of the embedding $\Ss_{\tria,0}^{0,1}$ into $\Ss_{\tria}^{-1,1}$.

For $0 \leq j \leq L$, let $\bm{R}_j$ be the  representation of the $L_2(\Omega)$-orthogonal projector of $\Ss_{\tria}^{-1,1}$ onto $\Ss_{\tria_{\new{j}}}^{-1,1}$, and let $\bm{R}_{-1}:=0$.

For $0 \leq j \leq L$, let $\bm{H}_j$ be the representation of the averaging operator $H_j\colon\Ss_{\tria_j}^{-1,1} \rightarrow \Ss^{0,1}_{\tria_j,0}$ defined by
\begin{equation} \label{101}
(H_j u)(\nu)=\frac{\sum_{\{T \in \tria_j\colon \nu \in \new{N_T}\}} |T| \,u|_T(\nu)}{\sum_{\{T \in \tria_j\colon \nu \in \new{N_T}\}} |T|},\quad(\nu \in N^0_\tria),
\end{equation}
and let $\bm{H}_{-1}:=0$.

For $1 \leq j \leq L$, let $\bm{P}_j$ be the representation of the embedding $\Ss_{\tria_{j-1},0}^{0,1} \rightarrow \Ss_{\tria_j,0}^{0,1}$ (\new{often called} \emph{prolongation}), and let $\bm{P}_0:=0$.

Then the representation $\bm{B}_\tria^{\Ss}$ of $B_\tria^{\Ss}$ from~\eqref{eq:multilevel_operator} is given by
\[
\bm{B}_\tria^{\Ss}=\bm{E}_\tria^\top \Big(\sum_{j=0}^L  (\bm{H}_j \bm{R}_j-\bm{P}_j \bm{H}_{j-1} \bm{R}_{j-1})^\top 2^{j (\frac{2s}{d}-1)} (\bm{H}_j \bm{R}_j-\bm{P}_j \bm{H}_{j-1} \bm{R}_{j-1})\Big)\bm{E}_\tria.
\]

\new{Applying $\bm{E}_\tria$ amounts to duplicating values at any internal node with a number equal to the valence of that node.}

\new{By representing $\tria$ as the leaves of a binary tree with roots being the simplices of $\tria_\bot$,  computing for $\vec{x} \in \ran \bm{E}_\tria$ the sequence $(\bm{R}_j \vec{x})_{0 \leq j \leq L}$ amounts to computing, while traversing from the leaves to the root, for any
parent and both its children the orthogonal projection of a piecewise linear function on the children to the space of linears on the parent. For $d=2$, the matrix representation of the latter projection is given in Figure~\ref{fig3}.
\begin{figure}
\begin{center}
\raisebox{-10ex}[1.3cm][1.3cm]{{\includegraphics[scale=0.5]{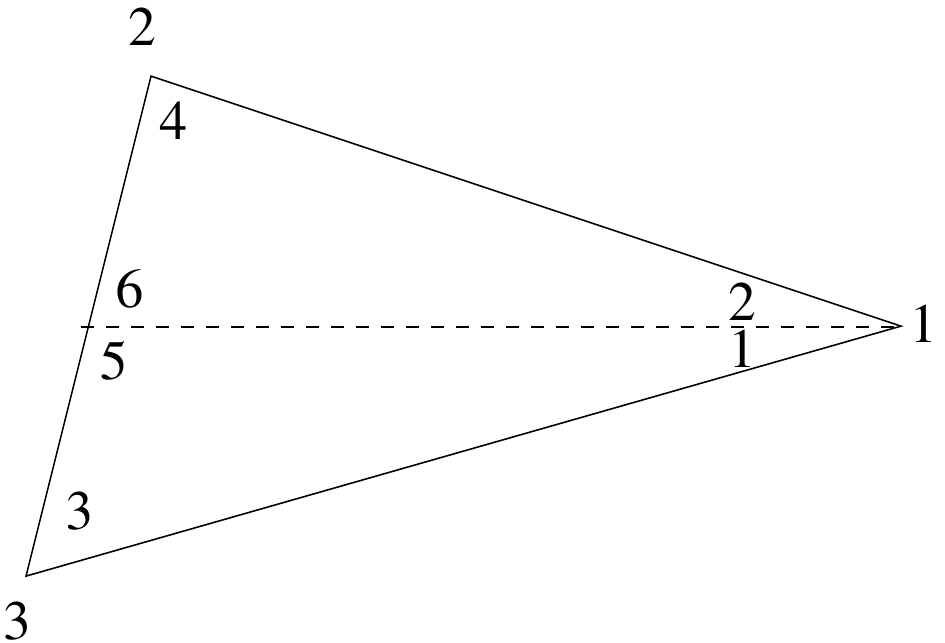}}}
\qquad $\left[
\begin{array}{@{}rrrrrr@{}}
\frac12 & \frac12 & 0 & 0 & 0 & 0\\
-\frac14& \frac14&-\frac14 & \frac34 &0 & \frac12\\
\frac14&-\frac14&\frac34&-\frac14 &\frac12 & 0\\
\end{array}
\right]$
\end{center}
\caption{\new{Numbering of the vertices of the parent and that of both children for $d=2$, and the resulting matrix representation of the orthogonal projection of the space of piecewise linears on the children to the space of linears on the parent.}}
\label{fig3}
\end{figure}}

\begin{proposition}\label{prop:linear_complexity}
The application of $\bm{B}_\tria^{\Ss}$ can be computed in ${\mathcal O}(\# \tria)$ operations.
\end{proposition}

\begin{proof}
\new{Because the number of nodes in a binary tree is less than 2 times the number of its leaves, for $\vec{x} \in \R^{\dim \Ss_{\tria,0}^{0,1}}$  the computation of  the sequence $(\bm{R}_j \bm{E}_\tria \vec{x})_{0 \leq j \leq L}$ takes ${\mathcal O}(\# \tria)$ operations.
From Proposition~\ref{propWZ} recall that any vector in $\ran \bm{H}_j \bm{R}_j-\bm{P}_j \bm{H}_{j-1} \bm{R}_{j-1}$ vanishes at $M^0_{\new{j}}$, so that the number of its non-zero entries is bounded by
$\# (\new{N^0_j} \setminus M^0_{\new{j}}) \leq 3 \#(\new{N^0_j} \setminus N^0_{\new{j-1}})$.
Knowing already $\bm{R}_j \bm{E}_\tria\vec{x}$ and $\bm{R}_{j-1} \bm{E}_\tria\vec{x}$, computing any non-zero entry
of $(\bm{H}_j \bm{R}_j-\bm{P}_j \bm{H}_{j-1} \bm{R}_{j-1})\bm{E}_\tria \vec{x}$ requires ${\mathcal O}(1)$ operations.}
\end{proof}

We conclude that the operator $B_\tria^{\Ss}$, with above matrix representation \new{$\bm{B}_\tria^{\Ss}$}, satisfies the \new{requirements} of Theorem~\ref{assumption}.

\section{Manifold case}\label{sec:manifold}
Let $\Gamma$ be a compact
 $d$-dimensional Lipschitz, piecewise smooth manifold in $\R^{d'}$ for some $d' \geq d$ with or without boundary $\partial\Gamma$.
 For some closed measurable $\gamma \subset \partial\Gamma$ and $s \in [0,1]$, let
 \[
\W:=[L_2(\Gamma),H^1_{0,\gamma}(\Gamma)]_{s,2},\quad \V:=\W'.
\]
 We assume that $\Gamma$ is given as the essentially disjoint union of $\cup_{i=1}^p \overline{\chi_i(\Omega_i)}$, with, for $1 \leq i \leq p$,
$\chi_i\colon \R^d \rightarrow \R^{d'}$ being some smooth regular parametrization, and $\Omega_i \subset \R^d$ an open polytope.
 W.l.o.g.~assuming that for $i \neq j$, $\overline{\Omega}_i \cap \overline{\Omega}_j=\emptyset$, we define
 \[
 \chi\colon \Omega:=\cup_{i=1}^p \Omega_i \rightarrow \cup_{i=1}^p \chi_i(\Omega_i) \text{ by } \chi|_{\Omega_i}=\chi_i.
 \]

Let $\bbT$ be a family of conforming partitions $\tria$ of $\Gamma$ into `panels' such that, for $1 \leq i \leq p$,  $\chi^{-1}(\tria) \cap \Omega_i$ is a uniformly shape regular conforming partition of $\Omega_i$ into $d$-simplices (that for $d=1$ satisfies a uniform $K$-mesh property).
We assume that $\gamma$ is a (possibly empty) union of `faces' of $T \in \tria$ (i.e., sets of type $\chi_i(e)$, where $e$ is a $(d-1)$-dimensional face of $\chi_i^{-1}(T)$).

  We set
 \begin{align*}
\V_\tria&:=\{u \in L_2(\Gamma)\colon u \circ \chi |_{\chi^{-1}(T)} \in \cP_0 \,\,(T \in \tria)\} \subset \V,\\
\intertext{or}
\V_\tria&:=\{u \in C(\Gamma)\colon u \circ \chi |_{\chi^{-1}(T)} \in \cP_1 \,\,(T \in \tria)\} \subset \V,\\
\intertext{equipped with canonical basis $\Xi_\tria$, and, for the construction of a preconditioner, }
\Ss^{0,1}_{\tria,0}&:=\{u \in H^1_{0,\gamma}(\Gamma)\colon u \circ \chi |_{\chi^{-1}(T)} \in \cP_1 \,\,(T \in \tria)\} \subset \W,
 \end{align*}
 equipped with canonical basis  $\Phi_\tria$.

As in the domain case, a space $\W_\tria \subset \W$ can be constructed with $\dim \W_\tria= \dim V_\tria$ and
$\inf_{\tria \in \bbT} \inf_{0 \neq v \in \V_\tria} \sup_{0 \neq w \in \W_\tria} \frac{\langle v,w\rangle_{L_2(\Gamma)}}{\|v\|_{\V} \|w\|_{\W}}>0$, which can be equipped with a locally supported basis $\Psi_\tria$ that, \new{modulo a scaling}, is $L_2(\Gamma)$-biorthogonal to $\Xi_\tria$.
Now assuming that a family of $B_\tria^{\Ss}\in \cLis(\Ss^{0,1}_{\tria, 0}, (\Ss^{0,1}_{\tria, 0})')$ uniformly is available, the construction of an optimal preconditioner $G_\tria$ follows exactly the same lines as outlined in Sect.~\ref{sec:preconditioning} for the domain case.

For the case that $\Gamma$ is not piecewise polytopal, a hidden problem is, however, that above construction of $\Psi_\tria$ requires exact integration
of lifted polynomials over the manifold. To circumvent this problem, in~\cite{249.97} we have relaxed the condition of $L_2(\Gamma)$-biorthogonality of $\Xi_\tria$ and $\Psi_\tria$ to
biorthogonality w.r.t.~to a mesh-dependent scalar product obtained from the $L_2(\Gamma)$-scalar product by replacing the Jacobian on the pull back of each panel by its mean. It was shown that the resulting preconditioner is still optimal, and that the expression for its matrix representation \new{(for the moment without the representation of $B_\tria^{\Ss}$)}, that was recalled in Sect.~\ref{sec:piecewise_contant} for the piecewise constant case, applies verbatim by \new{only} reading $|T|$ as the volume of the panel.\footnote{\new{In order to avoid the exact computation of this volume, actually it may read as $|\chi^{-1}(T)||\partial \chi(z)|$ for arbitrary $z \in \chi^{-1}(T)$.}}
\medskip

It remains to discuss the construction of an operator $B_\tria^{\Ss}$ of multi-level type, where it is now assumed that $\bbT$ is a family corresponding to newest vertex bisection.
An exact copy of the construction of $B_\tria^{\Ss}$ given in the domain case would require the application of the panel-wise $L_2(T)$-orthogonal projector
$Q_T$, cf. \eqref{eq:averaging}, which generally poses a quadrature problem.
Reconsidering the domain case, the proof of \cite[Lemma 3.7]{316.55} \new{(which provides the proof of the inequality `$\gtrsim$' in our Theorem~\ref{WZ})} builds on the fact that for $\tria_0 \prec \tria_1 \prec \cdots $ being a sequence of uniformly refined partitions,
the decomposition $\Ss_{\tria_L,0}^{0,1}=\sum_{j=0}^L \Ss_{\tria_{j},0}^{0,1} \cap (\Ss_{\tria_{j-1},0}^{0,1})^{\perp_{L_2(\Omega)}}$, where $\Ss_{\tria_{-1},0}^{0,1}:=\{0\}$,
is stable, uniformly in $L$, w.r.t.~the norm on $\W$.
This stability holds also true when the orthogonal complements are taken w.r.t.~a weighted $L_2(\Omega)$-scalar product, for any weight $w$ with $w,\,1/w \in L_\infty(\Omega)$.

This has the consequence that for the construction of the multi-level operator $B_\tria^{\Ss}$ in the manifold case, we may equip $L_2(\Gamma)$ with scalar product
\[
    \sum_{i=1}^p \int_{\Omega_{i}} u(\chi_{i} (x)) v(\chi_{i}(x)) \,dx,
\]
which is constructed from the canonical $L_2(\Gamma)$-scalar product by simply omitting the Jacobians $|\partial \chi_i(x)|$.
With this modified scalar product, the panel-wise orthogonal projector $Q_T$ is the same as in the domain case. We conclude that the resulting $B_\tria^{\Ss}$ as in \eqref{eq:multilevel_operator} is in $\cLis(\Ss_{\tria,0}^{0,1},(\Ss_{\tria,0}^{0,1})')$
uniformly, and that its application can be performed in linear complexity.
\new{Indeed, its implementation is equal as in the domain case as described in Sect.~\ref{sec:impl} when $|T|$ in \eqref{101} is read as $|\chi^{-1}(T)|$.}

\section{Numerical Experiments}\label{sec:numerics}
Let $\Gamma = \partial [0,1]^3 \subset \R^3$ be the two-dimensional manifold without boundary given as the boundary of the unit cube, $\W := H^{1/2}(\Gamma)$, $\V := H^{-1/2}(\Gamma)$. We consider the trial space $\V_\tria =\Ss_\tria^{-1,0}\subset \V$ of discontinuous piecewise constants. We will evaluate preconditioning of the discretized single layer operator $A_\tria \in \cLis(\V_\tria, \V_\tria')$.

The role of the opposite order operator in $\cLis(\Ss^{0,1}_{\tria,0},(\Ss^{0,1}_{\tria,0})')$ from Sect.~\ref{sec:theory} will be fulfilled by the multi-level operator $B_\tria^{\Ss}$ from~\eqref{eq:multilevel_operator}.
Equipping $\Ss^{0,1}_{\tria,0}$ with the nodal basis $\Phi_\tria$, the matrix representation of the preconditioner $G_\tria$ from Sect.~\ref{sec:theory} reads as
\[
\bm{G}_\tria=\bm{D}_\tria^{-1}\big(\bm{p}_\tria^\top \bm{B}_\tria^{\Ss} \bm{p}_\tria+\beta \bm{q}_\tria^\top \bm{D}^{1/2}_\tria \bm{q}_\tria\big)\bm{D}_\tria^{-1},
\]
for $\bm{D}_\tria=\diag\{|T|\colon T \in \tria\}$, uniformly sparse $\bm{p}_\tria$ and $\bm{q}_\tria$ as given in Sect.~\ref{sec:theory}, and with the representation of the multi-level operator $\bm{B}_\tria^{\Ss}$ given by
\[
\bm{B}_\tria^{\Ss}=\bm{E}_\tria^\top \Big(\sum_{j=0}^L  (\bm{H}_j \bm{R}_j-\bm{P}_j \bm{H}_{j-1} \bm{R}_{j-1})^\top 2^{ -j/2} (\bm{H}_j \bm{R}_j-\bm{P}_j \bm{H}_{j-1} \bm{R}_{j-1})\Big)\bm{E}_\tria,
\]
for the representations $\bm{E}_\tria, \bm{H}_j, \bm{R}_j$ and $\bm{P}_j$ as provided in Sect.~\ref{sec:impl} \new{(the minor adaptations in the manifold case described in Sect.~\ref{sec:manifold} to the matrix representations from Sections \ref{sec:theory} and \ref{sec:impl} vanish in the current simple case).}

The BEM++ software package~\cite{249.04} is used to approximate the matrix representation of the discretized single layer operator $\bm{A}_\tria$
by hierarchical matrices based on adaptive cross approximation~\cite{127.7, 19.896}.

\new{Equipping $\V_\tria$ and $\R^{\dim \V_\tria}$ with `energy-norms' $\sqrt{(A_\tria \cdot)(\cdot)}$ or $\|\bm{A}_\tria^{\frac12}\cdot\|$, respectively, we calculated the (spectral) condition numbers
$\kappa_{\cL(\V_\tria,\V_\tria)}(G_\tria A_\tria)=$\linebreak$\kappa_{\cL(\R^{\dim \V_\tria},\R^{\dim \V_\tria})}(\bm{G}_\tria \bm{A}_\tria)=\rho(\bm{G}_\tria \bm{A}_\tria)\rho((\bm{G}_\tria \bm{A}_\tria)^{-1})$ using the Lanczos method.}

As initial partition $\tria_\bot = \tria_1$ of $\Gamma$ we take a conforming partition consisting of $2$ triangles per side, so $12$ triangles in total, \new{with an assignment of the newest vertices}
that satisfies the matching condition.
\new{We fixed $\beta = 5.3$, being the value for which, for a relative small uniform refinement $\tria$ of $\tria_\bot$, we found $\rho(\bm{D}_\tria^{-1}\bm{p}_\tria^\top \bm{B}^{\Ss} _\tria \bm{p}_\tria \bm{D}_\tria^{-1} \bm{A}_\tria) = \rho(\bm{D}_\tria^{-1} \beta \bm{q}_\tria^\top \bm{D}_\tria^{1/2} \bm{q}_\tria \bm{D}_\tria^{-1}\bm{A}_\tria)$.}

\subsection{Uniform refinements}
Here we let $\bbT$ be the sequence $\{\tria_k\}_{k \geq 1}$ of (conforming) uniform refinements, that is, $\tria_k \succ \tria_{k-1}$ is found by bisecting each triangle from $\tria_{k-1}$ into $2$ subtriangles using Newest Vertex Bisection.

Table~\ref{table:unif} shows the condition numbers of the preconditioned system in this situation. The condition numbers are relatively small, and the timing results show that the implementation of the preconditioner is indeed linear.

\begin{table}[]
\centering
\caption{Spectral condition numbers of the preconditioned single layer system discretized by piecewise constants $\Ss_\tria^{-1,0}$, using uniform refinements.
Preconditioner $\bm{G}_\tria$ is constructed using the multi-level operator with $\beta = 5.3$.
The last column indicates the number of seconds per degree of freedom per application of $\bm{G}_\tria$.}
\begin{tabular}{@{}rrccc@{}}
\toprule
dofs     &   $\kappa_S(\bm{A}_\tria)$ & $\kappa_S(\bm{G}_\tria\bm{A}_\tria)$ & sec / dof \\
\toprule
$    12$&$  14.5$& $2.6$ &   $2.6\cdot10^{-5}$\\
$    48$&$  31.0$& $2.7$ &   $1.4\cdot10^{-5}$\\
$   192$&$  59.9$& $2.8$ &   $4.9\cdot10^{-6}$\\
$   768$&$ 118.7$& $3.3$ &   $1.4\cdot10^{-6}$\\
$  3072$&$ 234.6$& $3.8$ &   $6.3\cdot10^{-7}$\\
$ 12288$&$ 450.4$& $4.1$ &   $3.3\cdot10^{-6}$\\
$ 49152$&$ 852.5$& $4.3$ &   $6.5\cdot10^{-7}$\\
$196608$&$1566.4$& $4.5$ &   $7.3\cdot10^{-7}$\\
$786432$&$2730.5$& $4.6$ &   $7.8\cdot10^{-7}$\\
\bottomrule
\end{tabular}
\label{table:unif}
\end{table}

\subsection{Local refinements}
Here we take $\bbT$ as a sequence $\{\tria_k\}_{k \geq 1}$ of (conforming) locally refined partitions,
where $\tria_k \succ \tria_{k-1}$ is constructed by applying Newest Vertex Bisection to all triangles in $\tria_{k-1}$ that touch a corner of the cube.

Table~\ref{tbl:local} contains results for the preconditioned single layer operator discretized by piecewise constants $\Ss_\tria^{-1,0}$.
The preconditioned condition numbers are nicely bounded, and the timing results confirm that our implementation of the preconditioner is of linear complexity,
also in the case of locally refined partitions.

\begin{table}[]
\centering
\caption{Spectral condition numbers of the preconditioned single layer system discretized by piecewise constants $\Ss_\tria^{-1,0}$, using local refinements at each of the eight cube corners. Operator $\bm{G}_\tria$ is applied
    using the multi-level operator with $\beta = 5.3$. The second column is defined by $h_{\tria, min} := \min_{T \in \tria} \sqrt{|T|} $.
    The last column indicates the number of seconds per degree of freedom per application of $\bm{G}_\tria$.}
\begin{tabular}{@{}rlccc@{}}
\toprule
dofs     &  $h_{\tria,min}$ &$\kappa_S(\bm{G}_\tria\bm{A}_\tria)$ & sec / dof \\
\toprule
 $   12$ &  $1.4\cdot10^{0}$ & $ 2.63 $& $  2.5 \cdot 10^{-5}$ \\
$  336$ &  $8.8\cdot10^{-2}$ & $ 2.73 $& $  2.4 \cdot 10^{-6}$ \\
$  720$ &  $5.5\cdot10^{-3}$ & $ 2.91 $& $  1.8 \cdot 10^{-6}$ \\
$ 1104$ &  $3.4\cdot10^{-4}$ & $ 2.96 $& $  1.8 \cdot 10^{-6}$ \\
$ 1488$ &  $2.1\cdot10^{-5}$ & $ 2.99 $& $  2.2 \cdot 10^{-6}$ \\
$ 1872$ &  $1.3\cdot10^{-6}$ & $ 2.98 $& $  2.0 \cdot 10^{-6}$ \\
$ 2256$ &  $8.4\cdot10^{-8}$ & $ 3.00 $& $  2.3 \cdot 10^{-6}$ \\
$ 2640$ &  $5.2\cdot10^{-9}$ & $ 3.00 $& $  2.0 \cdot 10^{-6}$ \\
$ 3024$ &  $3.2\cdot10^{-10}$ & $ 3.01$& $  2.3 \cdot 10^{-6}$ \\
$ 3408$ &  $2.0\cdot10^{-11}$ & $ 3.01$& $  2.5 \cdot 10^{-6}$ \\
$ 3696$ &  $2.5\cdot10^{-12}$ & $ 3.01$& $  2.6 \cdot 10^{-6}$ \\
\bottomrule
\end{tabular}
\label{tbl:local}
\end{table}

\newcommand{\etalchar}[1]{$^{#1}$}

\end{document}